\documentclass[11pt,reqno]{amsart}
\usepackage{graphicx,amsmath,amsthm,verbatim,amssymb,lineno,eucal,titletoc,enumerate}
\usepackage{bbm}
\usepackage{subfigure}
\usepackage{hyperref}
\hypersetup{colorlinks}
\pdfoutput=1

\newcommand{\arxiv}[1]{{\tt \href{http://arxiv.org/abs/#1}{arXiv:#1}}}

\newcommand{\floor}[1]{\left\lfloor {#1} \right\rfloor}

\newcommand{\old}[1]{}
\newcommand{\moniker}[1]{{\em (#1)}}

\newtheorem{theorem}{Theorem}
\newtheorem{prop}[theorem]{Proposition}

\newtheorem{conjecture}[theorem]{Conjecture}

\newtheorem{question}[theorem]{Question}

\theoremstyle{remark}

\numberwithin{counter}{section}

\theoremstyle{definition}

\def\00{\mathbf{0}}

\def\d{\,\mathrm{d}}
\def\weakstar{\stackrel{*}{\rightarrow}}

\def\N{\mathbb{N}}
\def\Z{\mathbb{Z}}

\def\R{\mathbb{R}}
\def\EE{\mathbb{E}}
\def\Var{\mathrm{Var}}

\def\eps{\epsilon}
\def\zero{\mathbf{0}}
\def\one{\mathbf{1}}
\def\Stab{\mathcal{S}}
\def\Wake{\mathcal{W}}

\newcommand{\ARW}{\mathrm{ARW}}
\newcommand{\zagg}{\zeta_\mathrm{a}} 
\newcommand{\zth}{\zeta_{\mathrm{c}}} 
\newcommand{\zsta}{\zeta_{\mathrm{s}}}
\newcommand{\sleep}{{\tt{s}}}
\newcommand{\awake}{{\tt{a}}}

\begin{document}

\title[]{Universality Conjectures for Activated Random Walk}

\author[L. Levine]{Lionel Levine}
\address{Department of Mathematics, Cornell University, Ithaca, NY 14853.}
\email{levine@math.cornell.edu}

\author[V. Silvestri]{Vittoria Silvestri}
\address{Department of Mathematics, University of Rome La Sapienza, 00185 Rome, Italy..}
\email{silvestri@mat.uniroma1.it}

\date{June 14, 2023}

\begin{abstract}
Activated Random Walk is a particle system displaying Self-Organized Criticality, in that the dynamics spontaneously drive the system to a critical state. How universal is this critical state?  We state many interlocking conjectures aimed at different aspects of this question: scaling limits, microscopic limits, temporal and spatial mixing, incompressibility, and hyperuniformity.
\end{abstract}

\maketitle

\section{Introduction}

Many complex systems in nature are driven by a steady source of energy which builds up slowly and is released in intermittent bursts.  An example is the steady accumulation of pressure between continental plates, which is released in sudden bursts in the form of earthquakes. Wildfires, landslides, avalanches, and financial crises all have this character. 

In a famous 1987 paper \cite{BTW87}, Bak, Tang and Wiesenfeld proposed both a general mechanism for how such systems arise, and a prototypical example.  Their term for the mechanism was \textbf{self-organized criticality (SOC)}, and their example was a pile of sand resting on the surface of a table.  Given the current slope, $\zeta$, of the pile, we can sprinkle more sand on top and measure two things:
\begin{itemize}
\item How much sand falls off the table?
\item Does $\zeta$ tend to increase or decrease?
\end{itemize}
In experiments, one finds that the system drives itself to a \textbf{critical slope} $\zeta_c$: If the pile is too flat ($\zeta<\zeta_c$), then very little sand falls off the table and $\zeta$ increases as sand is added. If the pile is too steep ($\zeta>\zeta_c$), then 
the sprinkling causes a lot of sand to fall off the table, so that $\zeta$ \emph{decreases} as sand is added. No matter the initial sand profile, the effect of adding more sand is therefore to drive the pile toward its critical slope $\zeta_c$. 

This new idea of self-organization to criticality led initially to a lot of excitement (exemplified by Per Bak's ambitiously titled book \emph{How Nature Works}), 
but its success in making specific testable predictions has been modest so far. We do not know of a ``universal model of SOC'' in the way that, for instance, Brownian motion is a ``universal model of diffusion.'' 

\subsection{In search of a universal model of SOC}

The most intensively studied model of SOC is called the Abelian Sandpile. In this model, the pile of sand is a collection of indistinguishable particles on the vertices of a fixed graph, for example the $d$-dimensional cubic lattice $\Z^d$. When a vertex has at least as many particles as the number of neighbors in the graph ($2d$, in the case of $\Z^d$), it \emph{topples} by sending one particle to each neighboring vertex. As a result, some of those neighboring vertices may now have enough particles to topple, enabling some of their neighbors in turn to topple, and so on: an avalanche. Dhar \cite{Dha90} discovered a beautiful algebraic structure underlying this model. 

\emph{Abelian networks} \cite{BL1} form a larger class of SOC models. Among these, the Stochastic Sandpile \cite{Dha99c}, the Oslo model \cite{Fre93}, and 
Activated Random Walk \cite{HS04,RS11} (but not the original Abelian Sandpile!) seem to have some ``universality'' in the sense that when the system size is large, its behavior does not depend much on details like the initial condition, the boundary conditions, or the underlying graph. 
However, the meaning of ``universality'' is rarely spelled out.
The purpose of this survey is to state precisely several senses in which one of these models, Activated Random Walk, seems to be ``universal''.

Activated Random Walk (ARW) is a particle system with two species, active particles ($\awake$) and sleeping particles ($\sleep$) that become active if an active particle encounters them ($\awake+\sleep \rightarrow 2\awake$).  Active particles perform random walk at rate 1. When an active particle is alone, it falls asleep ($\awake \rightarrow \sleep$) at rate $\lambda$. A sleeping particle stays asleep until an active particle steps to its location.  The parameter $\lambda>0$ is called the \emph{sleep rate}. We denote this dynamics by ARW($\Z^d , \lambda )$.  

To draw out the analogy between ARW and sandpiles: The sleeping particles play the role of sand grains, the movement of the active particles plays the role of toppling, and the awakening of sleeping particles by active particles can trigger an avalanche in which many particles wake up. The density of particles in the system plays the role of the slope of the sandpile. Just as a pile of high slope can easily be destabilized by adding a single sand grain, an ARW configuration with a high density of sleeping particles can easily be destabilized by adding a single active particle.

\subsection{Plan of the paper}

We discuss ARW dynamics in six settings, differing in the initial condition, underlying graph, or boundary conditions. 
Our conjectures focus on the approach to criticality, and on shared properties of the corresponding critical states.  
Sections \ref{s:pointsource} and \ref{s:multiple} consider finite particle configurations in infinite volume. Our conjectures touch on the existence of scaling limits (Conjectures \ref{c.agg}, \ref{c.uniform}, \ref{c.quadrature}),  
	 microscopic limits (Conjectures \ref{c.micro}, \ref{c.bulkaggdensity}, \ref{c.microsame}), 
	 and extra symmetry acquired in the limit. 
	
In Section \ref{s:stationaryergodic} we discuss infinite (stationary ergodic) particle configurations on $\Z^d$. We conjecture existence of a microscopic limit as the threshold density is approached from below (Conjecture \ref{c.subcritical}). 

In Sections \ref{s:wiredchain}, \ref{s:freechain}, \ref{s:wakechain} we consider three different Markov chains on ARW stable configurations on a finite graph. 
The main themes here are
	temporal mixing (the system quickly forgets its initial condition: Conjectures \ref{c.cutoff},\ref{c.dense}), 
	spatial mixing (the boundary condition does not affect observables in the bulk: Question \ref{q.measures}, Conjecture~\ref{c.bulkdensity}), 
	and a slow-to-fast phase transition (Question~\ref{q.slowfast}, Conjecture~\ref{c.slowfastwake}). 

Sections \ref{s:hyperuniformity} and \ref{s:correlations} discuss statistical properties of these ARW systems: hyperuniformity (Conjectures \ref{c.hyperwired}, \ref{c.hyperfree}, Question~\ref{q.hyperwake}) and site correlations (Tables \ref{table.free},\ref{table.wired},\ref{table.pointsource}). 

Several conjectures on shared properties of critical or stationary states in the different settings are offered throughout the article (see Conjectures \ref{c.densities},\ref{c:hockey},\ref{c.freethreshold}, Question \ref{q.measures} and Proposition \ref{p:hockey}). 

 We conclude in Section \ref{s:contrasts} by contrasting the conjectured behavior of ARW with what is known about the Abelian Sandpile model.

\section{Point Source}
\label{s:pointsource}

\begin{figure}[h!]
\centering
\begin{tabular}{ccc}
    \includegraphics[scale=0.58]{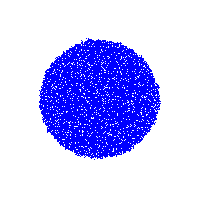} &
    \includegraphics[scale=0.58]{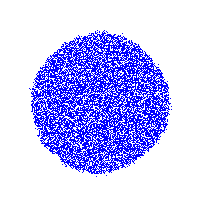} &
    \includegraphics[scale=0.58]{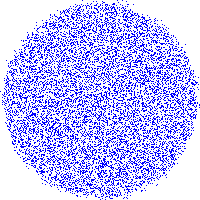} \\
    $\lambda=4$ & $\lambda=1$ & $\lambda=0.25$ \\
    $\zagg \approx 0.91$ & $\zagg \approx 0.68$ & $\zagg \approx 0.34$
\end{tabular}
\caption{\textsl{ARW aggregates formed by stabilizing a point source of $10000$ active particles at the origin in the square lattice $\Z^2$, at three different sleep rates $\lambda$. 
Each pixel represents a site of $\Z^2$: sites with a sleeping particle are colored blue, and empty sites are colored white. 
Particles spread farther when the sleep rate is lower, so the aggregate density $\zagg$ is an increasing function of $\lambda$.}}
\label{f:aggregates}
\end{figure}

\subsection{Spherical limit shape}

Consider $\ARW(\Z^d, \lambda)$ with initial configuration $n \delta_0$, consisting of $n$ active particles at the origin and all other sites empty. After a dynamical phase in which each particle performs random walk, and may fall asleep and be awakened many times, activity will die out when all particles fall asleep at distinct sites. 
We refer to the final configuration of $n$ sleeping particles as the \emph{ARW aggregate} (Figure~\ref{f:aggregates}).

\begin{conjecture} \moniker{Aggregate density $\zagg$} \label{c.agg}
Let $A_n$ denote the random set of sites visited by at least one walker during the dynamical phase of ARW started from $n$ particles at the origin in $\Z^d$.

There exists a positive constant $\zagg = \zagg(\Z^d,\lambda)$ such that for any $\eps>0$, with probability tending to $1$ as $n \to \infty$, the random set $A_n$ contains all sites of $\Z^d$ that belong to the origin-centered Euclidean ball of volume $(1-\eps)n/\zagg$; and $A_n$ is contained in the origin-centered Euclidean ball of volume $(1+\eps)n/\zagg$.
\end{conjecture}

A weak form of Conjecture~\ref{c.agg} in dimension $1$ is proved in \cite{LS2}. 
As the sleep rate $\lambda \uparrow \infty$, Activated Random Walk degenerates to Internal DLA, whose limit shape is proved to be a Euclidean ball \cite{LBG92}. 
The main barrier to applying Internal DLA methods is proving that sleeping particles are spread uniformly, which is the topic of our next conjecture.	

\subsection{Macroscopic structure of the aggregate}

An \emph{ARW configuration} in $\Z^d$ is a map \[ \eta : \Z^d \to \N \cup \{ \sleep \} \] 
where  $\eta (x) = \sleep $ indicates that there is a sleeping particle at $x \in \Z^d$, and $\eta (x) = k $ indicates that there are $k$ active particles at $x$.  
We write 
	\[ \Stab : (\N \cup \{\sleep\})^{\Z^d} \to \{0,\sleep\}^{\Z^d} \]
for the operation of \emph{stabilizing} an ARW configuration $\eta$: running ARW dynamics until all particles fall asleep.
\footnote{$\Stab(\eta)$ is always defined if $\eta$ has finitely many particles, but it may be undefined in general. The situation of infinite $\eta$ is discussed in Section~\ref{s:stationaryergodic}.}

Consider $\Stab ( n \delta_0)$, the ARW aggregate formed by stabilizing $n$ particles at the origin in $\Z^d$.
We will rescale the aggregate and take a limit as $n \to \infty$. For $x \in \R^d$ let
 	\[ a_n(x) = 1_{\big\{\Stab(n\delta_0) (\floor{n^{1/d} x})=\sleep\big\}}. \]
Write $f_n \weakstar f$ for weak-$*$ convergence: $\int f_n \phi \d x \to \int f \phi \d x$ for all bounded continuous test functions $\phi$ on $\R^d$, where $\d x$ is Lebesgue measure on $\R^d$.

\begin{conjecture} \moniker{Uniformity of the aggregate} \label{c.uniform}
The rescaled ARW aggregates $a_n$ satisfy
	\[ a_n \weakstar \zeta_a \one_B \]
with probability one, where $B$ is the origin-centered ball of volume $1/\zeta_a$ in $\R^d$.
\end{conjecture}

In other words, in the weak-$*$ scaling limit, the random locations of the sleeping particles in the aggregate blur out to a constant density $\zeta_a$ everywhere in the ball. 

\subsection{Microscopic structure of the aggregate}

\begin{figure}[h]
\includegraphics[width=.8\textwidth]{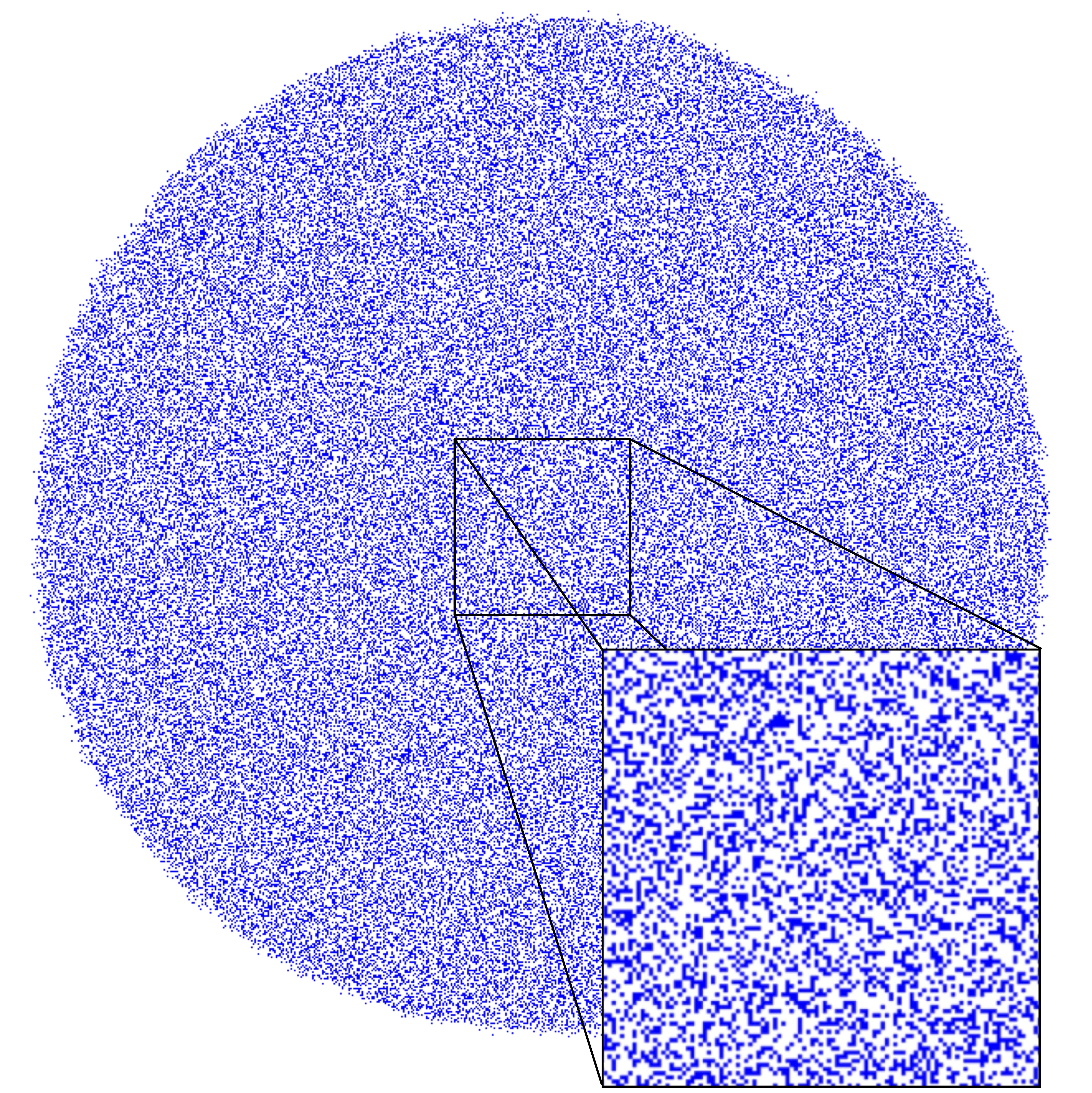}
\caption{An ARW aggregate of $100000$ particles in $\Z^2$ at sleep rate $\lambda = 0.25$, with a zoom-in of the microscopic structure deep inside.
\label{f:zoomin}
}
\end{figure}

The next conjecture zooms in to the fine scale random structure of the aggregate near the origin (Figure~\ref{f:zoomin}). 

Write $\alpha_n$ for the law of the aggregate $\Stab(n\delta_0)$. This is a probability measure on $\{0,\sleep\}^{\Z^d}$. We examine its marginals
\footnote{For a probability measure $\mu $ on $\{0,\sleep\}^{\Z^d}$ and a finite set $V \subseteq \Z^d$, we write  $\mu |_V $ for the marginal distribution on $\{0,\sleep\}^V$, that is 
 $ \mu |_V (\xi ) := \mu ( \{ \eta \in \{ 0,\sleep\}^{\Z^d} : \eta (v) = \xi (v) \; \forall v \in V \} )$, for $\xi \in \{0,\sleep\}^V$.
}
on a finite subset of $\Z^d$, as $n \to \infty$. 

\begin{conjecture} \label{c.micro} \moniker{Microscopic limit of the aggregate} 
For all finite $V \subset \Z^d$ and all $\xi \in \{ 0, \sleep \}^V $, the sequence $\alpha_n |_V(\xi)$ converges as $n \to \infty$.
\end{conjecture}

This conjecture would imply, by Kolmogorov's extension theorem, the existence of the infinite-volume limit 	
	\[ \alpha := \lim_{V \uparrow \Z^d} \lim_{n \to \infty} \alpha_n |_V, \]
which is a probability measure on the set of infinite stable configurations $\{0,\sleep\}^{\Z^d}$.   The outer limit is over an exhaustion of $\Z^d$, that is, a sequence of finite sets $V_1 \subset V_2 \subset \cdots$ such that $\bigcup_{n \geq 1} V_n = \Z^d$. To spell the limit out: For any finite $V \subset \Z^d$ and any configuration $\xi \in \{0,\sleep\}^V$,
	\[ \alpha_n |_V (\xi) \to \alpha |_V (\xi). \]
Note the order of limits: we are restricted to a fixed window $V$ as the size of the aggregate $n \to \infty$. Even though $\alpha$ is supported on configurations with an infinite number of particles, a sample from $\alpha$ is best imagined as tiny piece of an even larger aggregate!

\begin{conjecture} \label{c.translation}
The limit $\alpha$ is invariant with respect to translations of $\Z^d$.
\end{conjecture}

\begin{conjecture} \label{c.bulkaggdensity}
The limit $\alpha$ is supported on configurations of density $\zagg$.
\end{conjecture}

\section{Multiple sources}
\label{s:multiple}

\begin{figure}[ht]
\begin{tabular}{ccc}
\includegraphics[height=0.2\textheight]{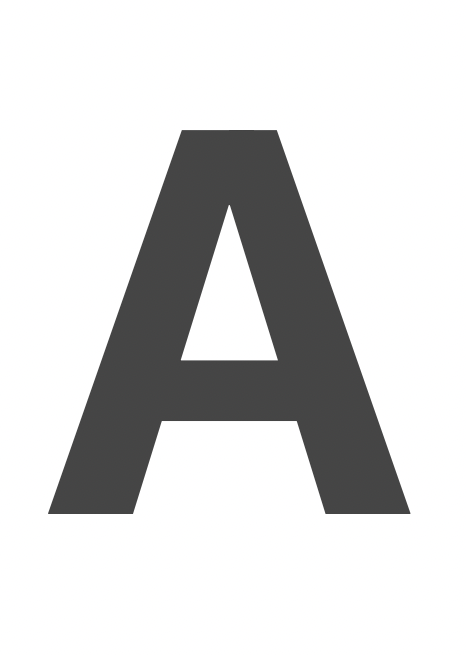} &
\includegraphics[height=0.2\textheight]{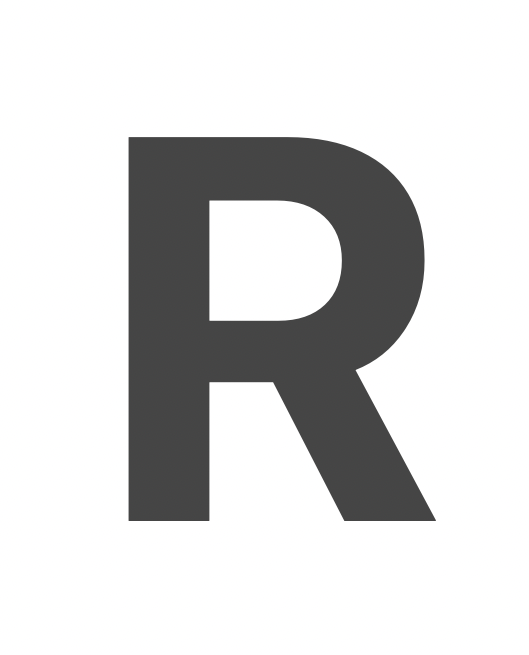} &
\includegraphics[height=0.2\textheight]{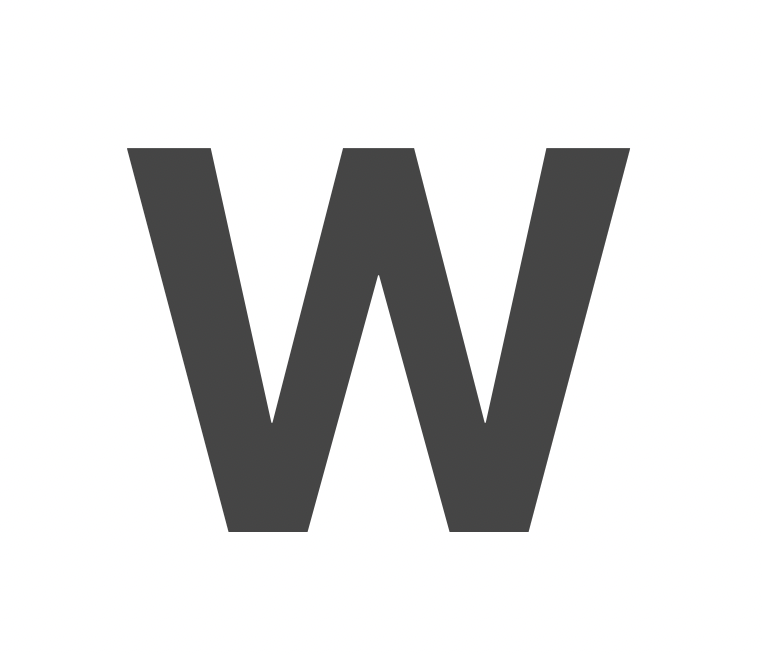} \\
$\downarrow \Stab_4$ & $\downarrow \Stab_1$ & $\downarrow \Stab_{\frac14}$ \\
\includegraphics[height=0.2\textheight]{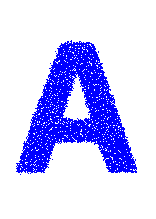} &
\includegraphics[height=0.2\textheight]{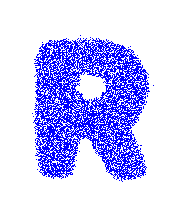} &
\includegraphics[height=0.2\textheight]{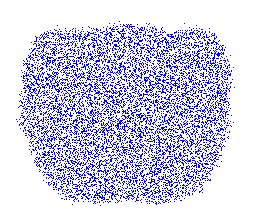} \\
 $\int_A u \geq 0.91 \int_{A^*} u$  &
 $\int_R u \geq 0.68 \int_{R^*} u$  &
 $\int_W u \geq 0.34 \int_{W^*} u$ 
 \end{tabular} 
 \caption{Top: The capital letters $A,R,W$ viewed as subsets of~$\R^2$.
 Middle: The stabilization of $1_{A \cap \eps \Z^2}$ at sleep rate 4, and $1_{R \cap \eps \Z^2}$ at sleep rate 1, and $1_{W \cap \eps \Z^2}$ at sleep rate $\frac14$. 
 Bottom: According to Conjecture~\ref{c.quadrature}, the resulting supports $A^*, R^*, W^*$ are characterized by quadrature inequalities for superharmonic functions. The constant appearing in the inequality is the aggregate density $\zagg$ ($\approx 0.91, 0.68, 0.34$ for sleep rates $4,1,\frac14$).}
\end{figure}

Let $A \subset \R^d$ be a bounded open set satisfying $\int_{\bar{A} \setminus A} \d x = 0$ where $\d x$ denotes $d$-dimensional Lebesgue measure.
 For $\eps>0$, let $\Stab (1_A \cap \eps \Z^d) $ be the configuration of sleepers that results from starting one active particle at each point of $A \cap \eps \Z^d$ and running activated random walk on $\eps \Z^d$ with sleep rate $\lambda$.  
The following conjectured scaling limit for ARW is inspired by Theorem 1.2 of \cite{LP10}, which describes the scaling limit of internal DLA in $\Z^d$.

\begin{conjecture} \moniker{Quadrature inequality} \label{c.quadrature}
As $\eps \to 0$,
	\[ \Stab (1_A \cap \eps \Z^d) \weakstar \zagg 1_{A^*} \]
where $A^*$ is the unique (up to measure zero) open subset of $\R^d$ satisfying
	\begin{equation} \label{eq:quadrature} \int_A u \d x \geq \zagg \int_{A^*} u \d x \end{equation}
for all integrable superharmonic functions $u$ on $A^*$.
\end{conjecture}

The intuition behind this conjecture is that the uniform density $1$ on $A$ spreads out to uniform density $\zagg <1 $ on the larger set $A^*$.  If $u$ is a superharmonic function on $A^*$, then the sum of the values of $u$ at all particle locations is approximately a supermartingale, leading to \eqref{eq:quadrature} by optional stopping. 
In the case of multiple point sources, $A^*$ is a smash sum of Euclidean balls \cite[Theorem 1.4]{LP10}.  

The next conjecture examines the microstructure of the aggregate near $\zero$.

\begin{conjecture} \moniker{Microstructure looks the same everywhere} \label{c.microsame}
Assume $\zero \in A^*$. For any finite $V \subset \Z^d$, the law of $\Stab (1_A \cap \eps \Z^d)|_{\eps V}$ has a limit as $\eps \to 0$, in the sense that for any configuration $\eta \in \{0,\sleep\}^V$
	\[ \mathbb{P} \big( \Stab (1_A \cap \eps \Z^d)(\eps v) = \eta(v) \text{ for all } v \in V \big) \to \alpha |_V (\eta) \]
 The limiting probability measure $\alpha$ on $\{0,\sleep\}^{\Z^d}$ is the same as in Conjecture~\ref{c.micro}. In particular, $\alpha$ does not depend on $A$.
\end{conjecture}


So far we have examined initial conditions with a finite number of particles only. The next section examines infinite configurations.

\section{Stationary Ergodic} \label{s:stationaryergodic}

For $\ARW(\Z^d, \lambda)$, start with a stationary ergodic configuration $\eta : \Z^d \to \N$, where all particles are initially active. Running ARW dynamics, will all particles fall asleep? If each site of $\Z^d$ is visited only finitely often, then we say that $\eta$ \emph{stabilizes}. Rolla, Sidoravicius, and Zindy proved the remarkable fact that stabilizing depends only on the mean number of particles per site  \[ \zeta := \mathbb{E} (\eta(\zero)) . \]

\begin{theorem} \label{t.RSZ} \moniker{Universality of threshold density $\zth$, \cite{RSZ}}
	There exists a constant $\zth = \zth(\Z^d, \lambda)$ such that if $\zeta < \zth$ then
	$\eta$ stabilizes with probability $1$, and if $\zeta > \zth$ then with probability $1$, $\eta$ does not stabilize.
\end{theorem}

\subsection{Approaching the threshold from below}

Theorem~\ref{t.RSZ} ensures that the stabilization $\Stab(\eta)$ is always defined if $\zeta < \zth$.
What happens to the microstructure of $\Stab(\eta)$ as $\zeta \uparrow \zth$?
Start with a stationary ergodic configuration $\eta_0 : \Z^d \to \N \cup \{\sleep\}$ with mean $\zeta_0 < \zth$, and sprinkle some extra active particles:  
Letting $(\xi_t(x))_{x \in \Z^d}$ be independent Poisson random variables with mean $t < \zth - \zeta_0$, the configuration $\eta_0 + \xi_t$ stabilizes with probability 1.

\begin{conjecture} \moniker{Universal limit of subcritical measures} \label{c.subcritical}
Fix $\lambda>0$ and 
let $\mu_t$ be the law of the ARW stabilization of 
$\eta_0+\xi_t$ with sleep rate $\lambda$. 
There exists a limiting measure
	\[ \mu := \lim_{t \uparrow \zth-\zeta_0} \mu_t \]
supported on configurations  $ \eta \in \{0,\sleep\}^{\Z^d} $ of density $\zth$. 
Moreover, $\mu$ depends only on $\lambda$ and not on the initial configuration $\eta_0$.
\end{conjecture}

\section{The wired Markov chain} \label{s:wiredchain}

Fix a finite set $V \subset \Z^d$, and consider the particle system $\ARW(V, \lambda )$ in which particles evolve as in ARW with sleep rate $\lambda$, with the additional rule that when a particle 
exits $V$ it is killed (i.e.\ removed from the system).  Fix $v \in V$. The ARW \emph{wired Markov chain} $(w_k)_{k \geq 0 }$ on the state space $\{0,\sleep \}^V$ has the update rule: add one active particle at $v$ and stabilize, i.e.
	\[ w_{k+1} = \Stab_V (w_k + \delta_v), \]
where $\Stab_V$ denotes ARW stabilization with killing of any particles that exit $V$.

\subsection{Stationary distribution}
The stationary distribution of the Markov chain $(w_k)_{k \geq 0}$ does not depend on the choice of the site $v$ where particles are added, as for different $v$ the Markov transition operators commute! The next result gives an efficient way to sample exactly from the stationary distribution of this chain.

Start with the configuration $\one_V$, consisting of one active particle on each site of $V$, and let the particles perform $\ARW(V , \lambda )$ until no active particles remain. Some particles exit the system, and the remaining particles fall asleep in $V$. 
Denote by $\Stab_V(\one_V)$ the resulting random configuration of sleepers. 

\begin{prop}[Exact sampling, \cite{LL}]
The law of $\Stab_V(\one_V)$ is the unique stationary distribution of the ARW wired Markov chain on $V$ with sleep rate $\lambda$. 
\end{prop}

\begin{figure}[ht]
\centering
\setlength{\tabcolsep}{3mm}
\begin{tabular}{ccc}
    \includegraphics[width=0.28 \textwidth]{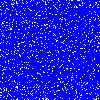} &
    \includegraphics[width=0.28 \textwidth]{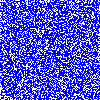} &
    \includegraphics[width=0.28 \textwidth]{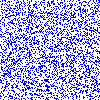} \\
    $\lambda=4$ & $\lambda=1$ & $\lambda=0.25$ \\
    $\zsta \approx 0.91$ & $\zsta \approx 0.68$ & $\zsta \approx 0.34$

\end{tabular}
\caption{\textsl{Stationary configurations for the ARW wired chain on a $100\times100$ box, at three different sleep rates $\lambda$. The stationary density $\zeta_s$ is an increasing function of $\lambda$.}}
\label{f:wired}
\end{figure}

For any given set $V \subset \Z^d$ let $\partial V $ denote its boundary and $\# V$ denote its cardinality. 
Write $|\Stab_V (\one_V)|$ for the total number of (sleeping) particles in $\Stab_V (\one_V)$. 

\begin{conjecture} \moniker{Stationary density $\zsta$} \label{c.stationary}
Then there exists  a constant $\zsta = \zsta(\Z^d,\lambda)$ such that for any exhaustion $V_1 \subset V_2 \subset \cdots \subset \Z^d$ satisfying $\# ( \partial V_n ) / \# V_n \to 0$ as $n\to\infty $, 		\[ \lim_{n\to\infty } \frac{|\Stab_{V_n} (\one_{V_n})|}{\# V_n} = \zsta \]
in probability. 	

\end{conjecture}
\begin{conjecture} \label{c.densities}
The critical densities from Sections \ref{s:pointsource}, \ref{s:stationaryergodic} and \ref{s:wiredchain} coincide: 
	\[\zagg   = \zeta_c =  \zsta  . \]
\end{conjecture}

\subsection{Infinite volume limit} \label{s:defpi}
Let $\pi_V$ denote the stationary distribution of the ARW wired Markov chain, as defined above. 
For a subset $W \subset V$,  write $\pi_{V} |_W $ for the restriction of $\pi_V$ to $W$. 

\begin{conjecture} \label{c.infinitevolume}
For any fixed finite set $W \subset \Z^d$, the measures $\pi_V |_W$ have a limit as $V \uparrow \Z^d$, and this limit does not depend on the exhaustion of $\Z^d$.  
\end{conjecture}


This conjecture would imply, by Kolmogorov's extension theorem, the existence of a limiting probability measure 
	\begin{equation} \label{eq:infinitevolume} 
	\pi  = \lim_{W \uparrow \Z^d} \lim_{V \uparrow \Z^d} \pi_{V} |_W. 
	\end{equation}
on the space of infinite stable configurations $w : \Z^d \to \{0,\sleep\}$.  We can then ask how this limit relates to the measures  $\alpha$ and $\mu$ from Sections \ref{s:pointsource} and \ref{s:stationaryergodic} above.

\begin{question}
\label{q.measures}
Is $\pi = \mu = \alpha$ ?
\end{question}

Can the wired boundary condition be felt deep inside $V$? We conjecture that as $V \uparrow \Z^d$, the particle density deep inside $V$ coincides with the overall density $\zeta_s$.

\begin{conjecture} \label{c.bulkdensity}
For $w \sim \pi$ we have
$\pi \{ w(\zero) = \sleep \} = \zeta_s$.
\end{conjecture}

\subsection{The hockey stick conjecture}
Write $(w_k)_{k \geq 0}$ for the ARW wired chain on the box $V := [1,L]^d \subset \Z^d$ with initial state $w_0=0$ (all sites are empty) and with \emph{uniform driving}: instead of adding particles at a fixed vertex, we add them at a sequence of independent vertices $v_1, v_2, \ldots$ with the uniform distribution on $V$:
	\[ w_{k+1} = \Stab_V (w_k + \delta_{v_{k+1}}), \qquad k \geq 0. \]

When does the wired chain begin to lose a macroscopic number of particles at the boundary?  A theorem of Rolla and Tournier partially answers this question. Define 
	\[ \zeta_w :=  \inf \Big\{ t>0 \,:\, \limsup_L  \frac{\EE ( |w_{tL^d}| )}{L^d} < t \Big\} \]
where, as usual, $|w_k|$ denotes the number of particles in $w_k$.

\begin{theorem}\label{th:RT}
 \cite[Proposition 3]{RT} $\zeta_w \geq \zeta_c$. 
 \end{theorem}

We conjecture $\zeta_w = \zth$, and that the stabilized density has the following simple piecewise linear form.

\begin{conjecture}[Hockey stick] \label{c:hockey} The wired chain on $[1,L]^d$ with uniform driving satisfies
\[ \frac{|w_{tL^d}|}{L^d} \to \begin{cases}  t, & t \leq \zth  \\
				\zeta_c, & t \geq \zth
				\end{cases}  \]
in probability as $L \to \infty$, where $\zth$ is the threshold density of Theorem~\ref{t.RSZ}.
\end{conjecture}

The name for this conjecture comes from the graph of the piecewise linear limit, which has the shape of a hockey stick (Figure~\ref{f.hockey}).

\begin{figure}[h!]
\includegraphics[width=.85\textwidth]{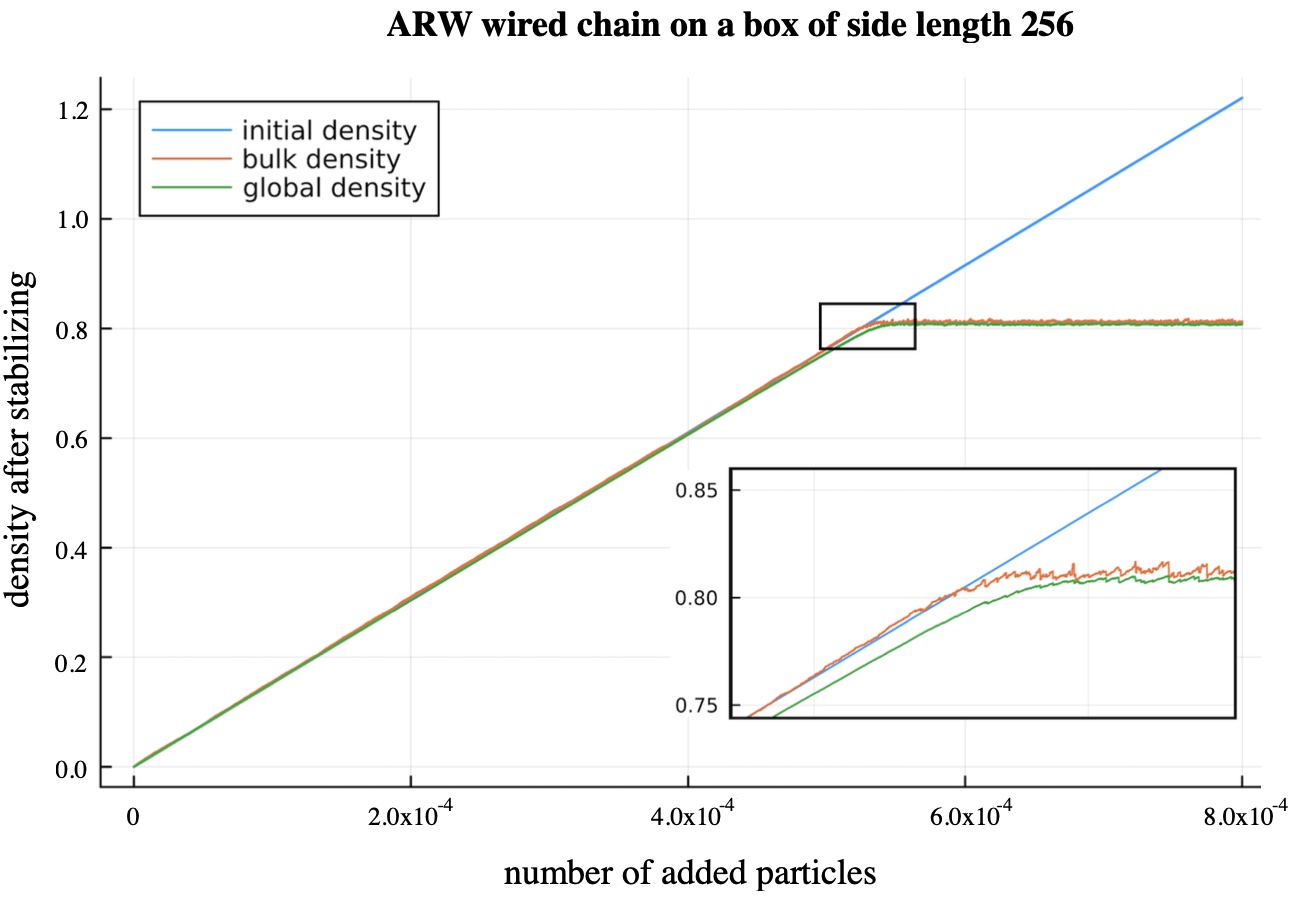}
\caption{The hockey stick: As particles are added, the density of the ARW wired chain increases to $\zeta_c$ and then flatlines. Here $V \subset \Z^2$ is a box of side length $L=256$, the sleep rate is $\lambda=2$, and $\zeta_c \approx 0.813$. ``Global density'' is the total number of particles divided by $L^2$. ``Bulk density'' is the number of particles in a central window of side length $L/2$, divided by $(L/2)^2$.
\label{f.hockey}
}
\end{figure}

\begin{prop}\label{p:hockey}
If Conjecture \ref{c:hockey} holds, then $ \zeta_w = \zth $. 
\end{prop}
\begin{proof}
By Theorem \ref{th:RT}, it suffices to prove that $ \zeta_w \leq \zth$. We argue by contradiction: suppose that $\zth < \zeta_w$. Then there is a $\zeta \in ( \zth , \zeta_w )$ such that both 
	\begin{equation}\label{eq1}
	\limsup_L \frac{\EE ( |w_{\zeta L^d}| )}{L^d} \geq  \zeta 
	\end{equation}
since $\zeta < \zeta_w$, and for any $\eps >0$ 
	\begin{equation}\label{eq2}
	\mathbb{P} \bigg( \bigg| \frac{ |w_{\zeta L^d}| }{L^d} - \zth \bigg| \geq \eps \bigg) \to 0 \qquad \mbox{ as } L \to\infty ,  
	\end{equation}
since $\zeta > \zth $ and Conjecture \ref{c:hockey} holds true. But by \eqref{eq1} there exists a diverging subsequence $(L_k)_{k\geq 1}$ such that 
	\[ \lim_{k\to \infty } \frac{\EE ( |w_{\zeta L_k^d}| )}{L_k^d} \geq  \zeta , \]
and since each term in this sequence is at most $\zeta $ by construction, it must be that the above limit actually equals $\zeta$. It follows that for any $\eps >0$ 
	\[ \lim_{k\to\infty } \mathbb{P} \bigg( \bigg| \frac{ |w_{\zeta L_k^d}| }{L_k^d} - \zeta \bigg| \geq \eps \bigg) 
	= \lim_{k\to\infty } \mathbb{P} \bigg( \zeta - \frac{ |w_{\zeta L_k^d}| }{L_k^d}  \geq \eps \bigg) 
	\leq \lim_{k\to\infty } \frac 1 \eps \bigg( \zeta - \frac{\EE ( |w_{\zeta L_k^d}| )}{L_k^d}
	\bigg) = 0 .  \]
This contradicts \eqref{eq2} by uniqueness of the limit in probability. 
\end{proof}

\subsection{Fast mixing}
\label{s.mixing}

Consider the ARW wired chain in a discrete Euclidean ball $V = \{x \in \Z^d \,:\, \sum x_i^2 < L^2 \}$ with uniform driving.  We highlight the following conjecture from \cite{LL} to the effect that this chain mixes immediately after reaching the stationary density $\zsta$. 

\begin{conjecture} \label{c.cutoff} \moniker{Cutoff, \cite{LL}}
The ARW wired chain has cutoff in total variation at the time
	\[ t_{mix} = \zsta \# V. \]
\end{conjecture}

In \cite{LL} it is shown that $t_{mix} \leq (1+o(1))\# V$. The proof uses a coupling between Activated Random Walk and Internal DLA. Bristiel and Salez \cite{BS} show that the relaxation time is much smaller: $O(L^{d-1})$ in dimensions $d \neq 2$ and $O(L \log L)$ in dimension $2$. They also prove separation cutoff at time $\#V$. 

\subsection{Incompressibility}

A recurring challenge in proving several of the above conjectures is to show that ``dense clumps'' are unlikely.  We conjecture that clumps denser than the mean in the infinite-volume stationary state $w  \sim \pi$ have exponentially small probability.
Write $|w|_L := \sum_{x \in [1,L]^d} 1_{\{w(x)=\sleep\}}$ for the number of particles in the cube $[1,L]^d$. 

\begin{conjecture} \label{c.compress} \moniker{Incompressibility}
For each $\zeta > \zsta$, there is a constant $c= c(\zeta,\lambda) > 0$ such that for $w  \sim \pi $
\[ P( |w |_L \geq \zeta L^d ) \leq \exp (-c L^d).  \]
\end{conjecture}
The ideas introduced in \cite{AFG,FG} may be useful in proving incompressiblity for $\zeta$ sufficiently close to 1.

\section{The free Markov chain} \label{s:freechain}

Fix a finite connected graph $V$, an initial configuration $\phi_0: V \to \{0,\sleep\}$, and let
	\[ \phi_{k+1} = \Stab (\phi_k + \delta_{v_{k+1}}) \]
be the configuration of sleeping particles obtained by adding one active particle at a random vertex $v_k$, and then stabilizing by ARW dynamics in $V$ with sleep rate $\lambda$.  The vertices $v_1, v_2, \ldots$ are independent with the uniform distribution on $V$.

Unlike the ARW wired Markov chain in Section \ref{s:wiredchain}, particles cannot escape $V$. So the total number of particles is deterministic: $|\phi_k| = |\phi_0| +k$.  As long as this number does not exceed $ \#V$,  stabilization happens in finite time, but if the number of particles is large then it could take a long time (even exponentially long, \cite{expo})!  We will define the \emph{threshold time} as the first time $k$ such that $\phi_k$ takes ``too long'' to stabilize. 

Let $(\phi_k)_{k\geq 0}$ denote the ARW free chain on $V$ initiated from the empty configuration $\phi_0 = 0$. 
For $k\geq 0$, denote by  $U_k$ be the total number of random walk steps needed to stabilize $\phi_k + \delta_{v_{k+1}}$. 
For any function $f: \N \to \R$, let 
	\[ \tau_f(V) = \inf \{k \geq 0 \,:\, U_k \geq f( \# V ) \}. \]

\begin{conjecture} \moniker{Concentration of the threshold time} \label{c.freethreshold}
Let $V = \Z_L^d $ be the $d$-dimensional torus of side-length $L$. 
There exists a superlinear function $f:\N \to \R$ such that, as $L \uparrow \infty$, 
	\[ \lim_{L \uparrow \infty}  \frac{\tau_{f} (\Z_L^d)}{L^d} = \zth \]
in probability, 
where $\zth$ is the threshold density of Section \ref{s:stationaryergodic}.
\end{conjecture}

A stronger formulation would posit a sharp transition from linear to exponential time:

\begin{question} \label{q.slowfast}
Is it true that
	\[ U_{tL^d} = \begin{cases} O(L^d),  &t<\zeta_c  \\
						 \exp(\Omega(L^d)),   &t>\zeta_c  ?
				\end{cases}
	\] 
\end{question}

\section{The Wake Markov Chain}
\label{s:wakechain}

Fix a finite connected graph $V$ and an initial configuration $\varphi_0: V \to \{0,\sleep\}$ with $|\varphi_0 | \leq \# V $.  Let $\Wake$ denote the operator that acts on stable particle configurations on $V$ by waking all particles up. 
The ARW wake Markov chain, supported on stable particle configurations on $V$, is defined by 
	\[ \varphi_{k+1}= \Stab ( \Wake ( \varphi_k ) ) .  \]
So in one time step of the wake chain, we wake all particles up and then stabilize. 
Note that stabilization is always possible, though it may take a long time, since $| \varphi_k | = |\varphi_0| \leq \# V $ for all $k\geq 0$.  

\subsection{Stationary measure} 
Take $V = \Z_L^d $ and let $\nu_{L,\zeta}$ denote the stationary measure of the ARW wake chain on $V$. Denote further by    $\tilde{\nu}_{L,\zeta}$ the law of the ARW free chain $(\phi_k)_{k\geq 0}$ at time $k = \zeta L^d $. Note that sleepers configurations drawn according to $\nu_{L,\zeta}$ and $\tilde{\nu}_{L,\zeta}$ have the same density $\zeta$ (in fact, the same number of particles). It is natural to conjecture that in the supercritical regime these measures are close.
\begin{conjecture} \label{c.freewake}
If $\zeta > \zeta_c$ then $d_{\textrm{TV}} \big( \nu_{L,\zeta} , \tilde{\nu}_{L,\zeta} \big) = o(1)$ as $L \to \infty$, where $d_{\textrm{TV}}$ denotes the Total Variation distance. 
\end{conjecture}
This would follow from the following, more general conjecture. 
\begin{conjecture} \moniker{Dense stabilized configurations are hard to distinguish} \label{c.dense}
Fix the dimension $d$ and sleep rate $\lambda$, and let $\zeta_c$ be the threshold density of Theorem~\ref{t.RSZ}.
For each $\zeta > \zeta_c$ and $\eps>0$ there is an $L_0$ such that for all $L \geq L_0$ and any two configurations $\eta, \tilde{\eta}$ of active particles on the discrete torus $\Z_L^d$ with $|\eta| = |\tilde{\eta}| \geq \zeta L^d$, their stabilizations $\Stab(\eta), \Stab(\tilde{\eta})$ satisfy
	\[ d_{\mathrm{TV}} (\Stab(\eta), \Stab(\tilde{\eta})) < \epsilon. \] 
\end{conjecture}

The underlying mechanism here is that the system takes a long time to stabilize. In particular, it is known that for small enough sleep rate, stabilization takes exponentially many steps in $L$ with high probability: This was proved in dimension $1$ by Basu, Ganguly, Hoffman and Richey \cite{expo}, and recently in all dimensions by Forien and Gaudilli\`{e}re \cite[Theorem 3]{FG}. Their result plus a coupling argument proves Conjecture~\ref{c.dense} for $\lambda$ small enough: One can couple the trajectories of any two particles in the ARW systems starting from $\eta$ and $\tilde{\eta}$ so that they will meet prior to stabilization with high probability. Provided all these couplings are successful, the processes stabilize to the same configuration $\Stab(\eta)= \Stab(\tilde{\eta})$.

\begin{figure}
\includegraphics[width=\textwidth]{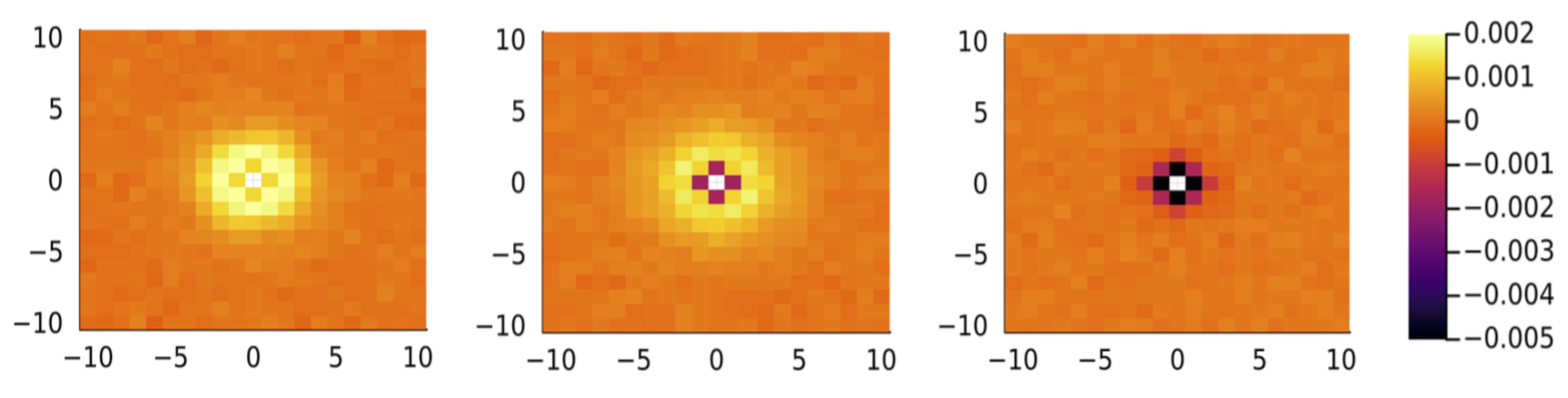}
\caption{Site covariance of the ARW Wake Markov Chain on the torus $\Z_{L} \times \Z_{L}$ with $L=2501$, sleep rate $1$, and subcritical density $0.3$ (much less than $\zeta_c \approx 0.68$). Left to right: covariances after $k=1,2,10$ steps of the wake chain from a uniform random initial condition with $\floor{0.3 L^2}$ particles.
Each non-central site is shaded according to the covariance between the events that a sleeping particle is located at that site and at the central site.
Initial positive correlations with nearby sites become negative after a few time steps.
\label{fig.wake1}
}
\end{figure}

\subsection{Mixing time}
How long does it take for the ARW wake chain to reach stationarity? We conjecture a transition from slow to instantaneous mixing at the threshold density $\zeta_c$.

\begin{conjecture}\label{c.slowfastwake}
Let $\eta$ be any stable configuration on the $d$-dimensional cycle $\mathbb{Z}_L^d$ on $L^d$ vertices, and  denote by $\zeta = |\eta|/L^d $ its particle density. Then the total variation mixing time of the ARW wake chain starting from $\eta$ is $1$ if $\zeta > \zeta_c $, while it is $\Omega (L^2 ) $ if $\zeta < \zeta_c $, where $\zeta_c$ is the threshold density of Theorem~\ref{t.RSZ}. 
\end{conjecture}
The first part of this conjecture, fast mixing at high density, would follow directly from Conjecture \ref{c.dense}.

\section{Hyperuniformity}
\label{s:hyperuniformity} 
Experiments suggest that the stationary states for the Activated Random Walk Markov chains introduced in Sections \ref{s:wiredchain},\ref{s:freechain},\ref{s:wakechain} above are \emph{hyperuniform}.
\subsection{The wired chain}
 For a random configuration $\eta \in \{ 0 , \sleep \}^{\Z^d}$ with  $\eta  \sim \pi $, write 
$|\eta |_{L} $ for the total number of particles in the cube $[1,L]^d$.

\begin{conjecture} 
\label{c.hyperwired}
Under $\pi$, the variance of $|\eta |_L$ is $O(L^\alpha)$ as $L \to\infty$,  for some $\alpha<d$.
\end{conjecture}

\begin{figure}[ht]
\includegraphics[width=.88\textwidth]{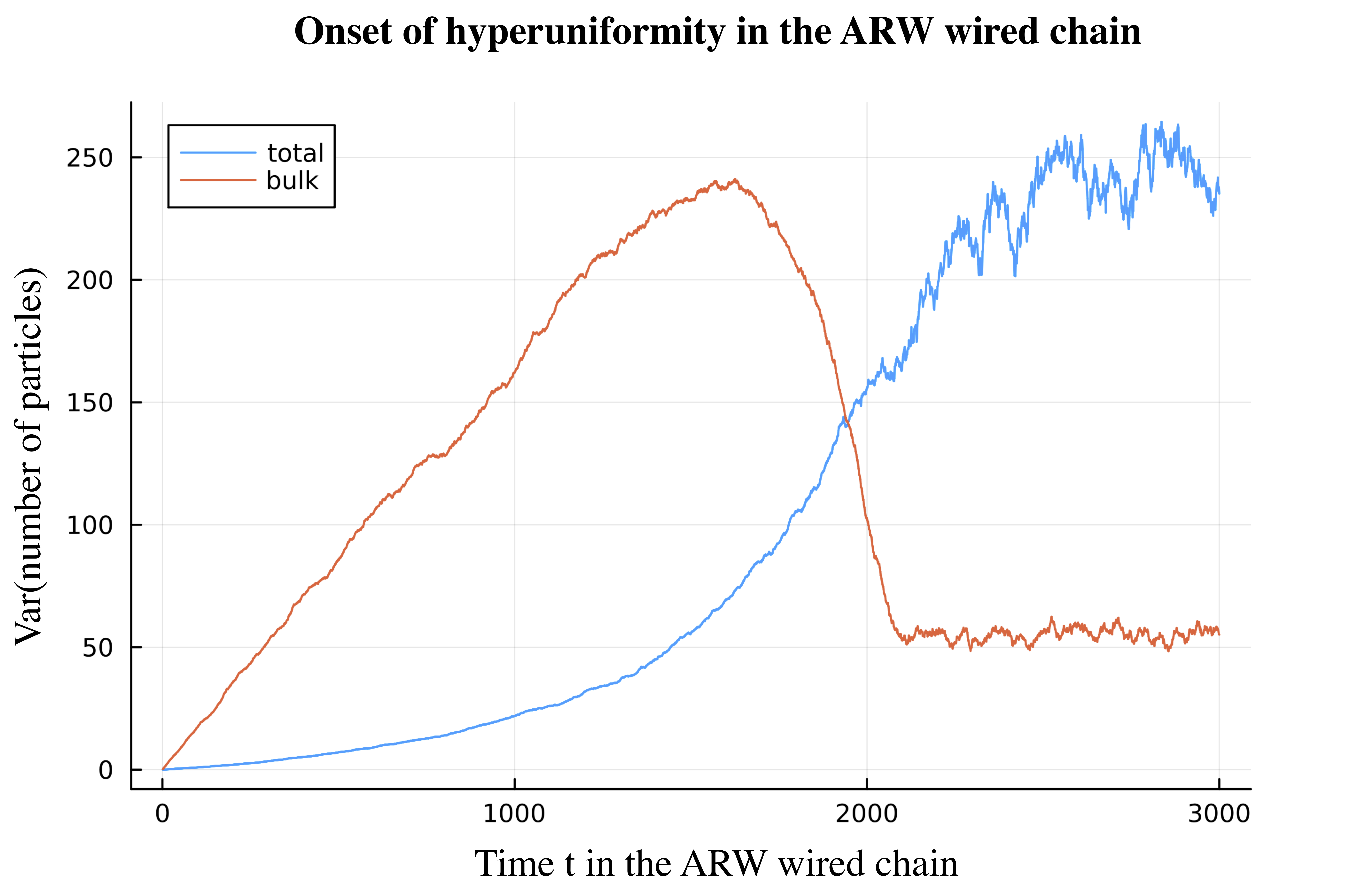}
\caption{The onset of hyperuniformity in the ARW wired chain on a square of side length $L=50$ at sleep rate $\lambda=2$. The variance $\EE |\eta_t|^2 - (\EE |\eta_t|)^2$ of the total number of particles in the system increases and then levels off after time $t = \zeta_s L^2$.  The variance of the number of particles in the bulk (central square of side length $L/2$) peaks and then drops steeply as $t/L^2$ approaches the stationary density $\zeta_s \approx 0.81$. 
}
\end{figure}

For comparison, observe that if the particles are placed in the box $[1,L]^d$ in an i.i.d.\ fashion, the variance is of order $L^d$. Thus hyperuniformity implies a kind of rigid repulsion among particles: in order to make the variance of the number of particles in the box $ [1,L]^d$ grow sublinearly with its volume $L^d$, the particle counts $\eta (v)$ for $v \in [1,L]^d$ must have significant negative correlations \cite{GL}. 

Burdzy has proved hyperuniformity in a related particle system called the Meteor Model \cite{Burdzy}. The challenge in adapting his proof to ARW lies in adapting the i.i.d.\ driving of the meteors to the correlated driving that results from active particles waking sleeping particles. \\

\subsection{The free chain}
For the free chain the number of particles increases by one at each time step, and we expect hyperuniformity to manifest starting at the threshold density $\zeta_c$. 
To state a hyperuniformity conjecture for the free chain, we will count particles in a box $B \subset \Z_L^d$.
Write $(\phi_k)_{k \geq 0}$ for the ARW free chain on the torus $\Z_L^d$, and write $|\phi_k|_B$ for the total number of particles in $B$ at time $k$.

\begin{conjecture}\label{c:hyperunif_free} \moniker{Onset of hyperuniformity in the free chain} \label{c.hyperfree}
There exists $\eps>0$ such that for any box $B = [0,\ell_1-1] \times \cdots \times [0,\ell_d-1] \subset \Z_L^d$ we have
	\[ \Var(|\phi_{\zeta L^d}|_B) = \begin{cases} \Theta(\ell_1\ldots\ell_d), & \zeta<\zeta_c \\
								O((\ell_1\cdots \ell_d)^{1-\eps}), & \zeta \geq \zeta_c. 
						\end{cases}
	\]
The implied constants depend only on $d,\lambda,\zeta$.
\end{conjecture}

\begin{figure}[h!]
\includegraphics[width=.85\textwidth]{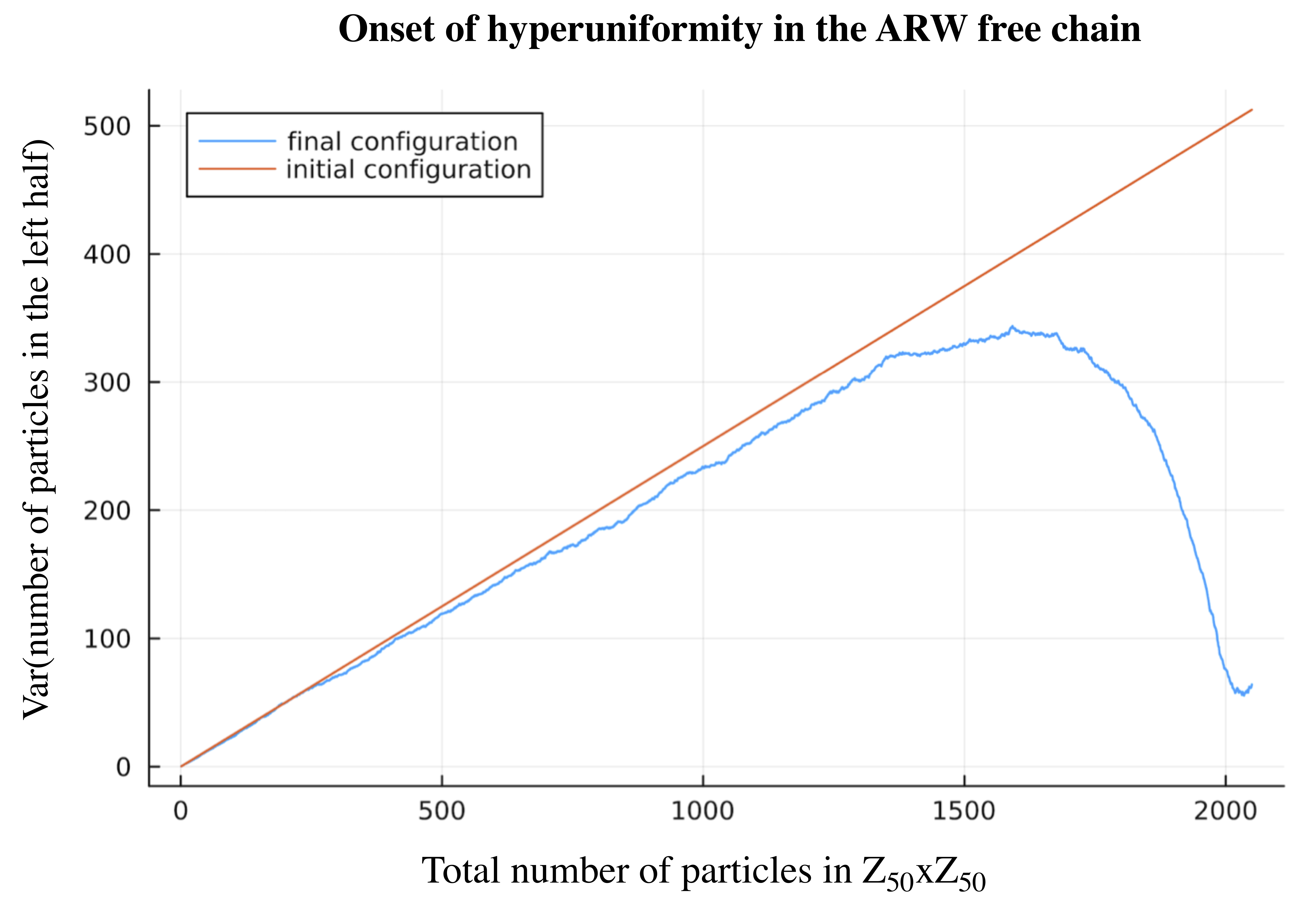}
\caption{
The onset of hyperuniformity in the ARW free chain on the discrete torus $\Z_{L} \times \Z_{L}$ with $L=50$ and sleep rate $\lambda=2$. The initial configuration of $k$ active particles at independent uniformly distributed sites, stabilizes to a final configuration of $k$ sleeping particles on the torus. The variance of the number of sleeping particles in the left half $\Z_{L/2} \times \Z_{L}$ initially increases with $k$, then peaks around $k =0.66 L^2$, and bottoms out around $\zeta_c L^2$ where $\zeta_c \approx 0.81$.
\label{f.hyperunif_free}
}
\end{figure}

\subsection{The wake chain}

\begin{figure}[ht]
\includegraphics[width=.8\textwidth]{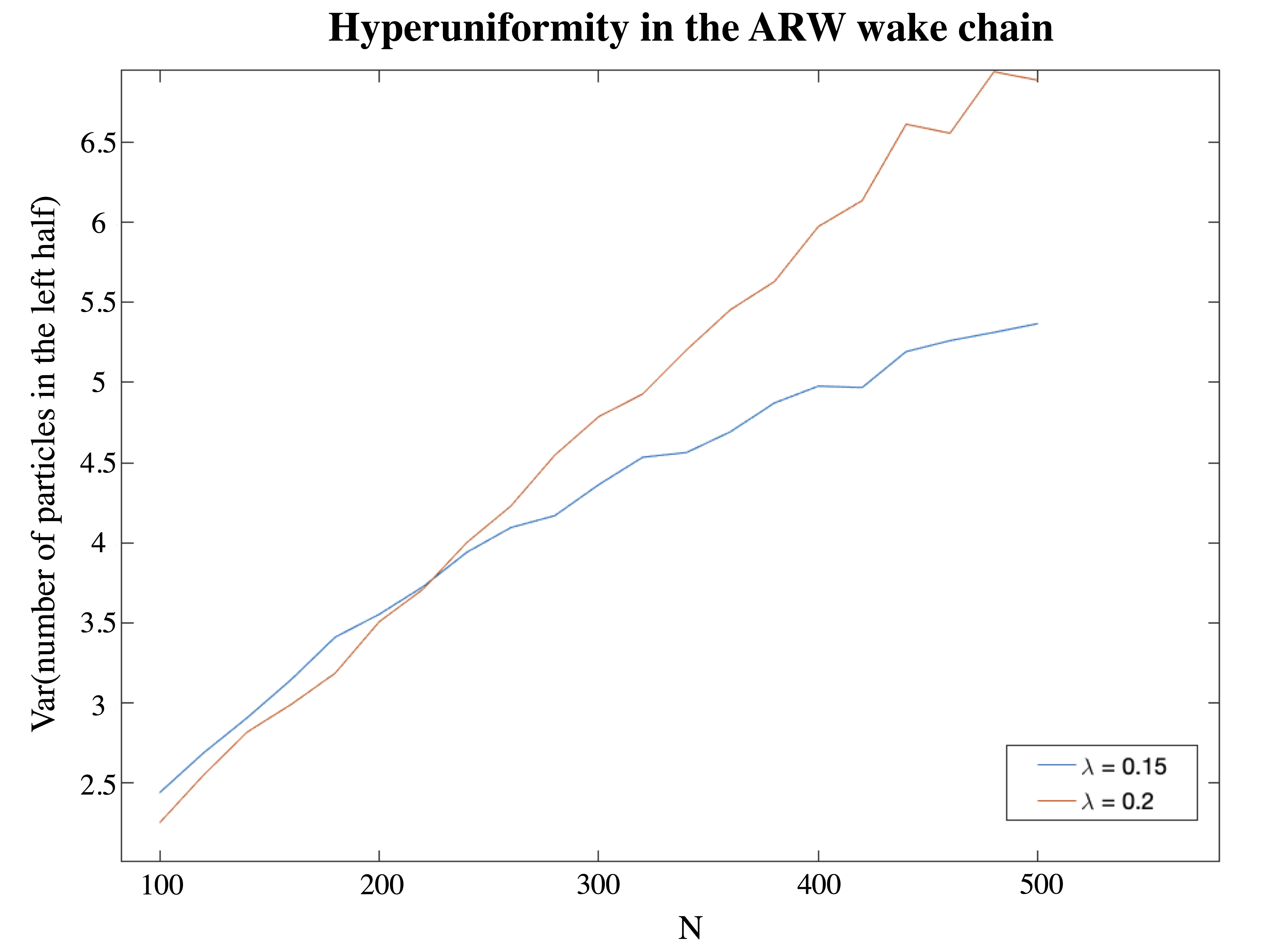}
\caption{The ARW wake chain on the cycle $\Z_N$ starting with $N/2$ particles at uniformly random sites and run for $100N$ time steps. The variance of the number of particles on the left half of the cycle grows linearly with $N$ at the subcritical sleep rate $\lambda = 0.2$, but sublinearly with $N$ at the critical sleep rate $\lambda = 0.15$.   (Here $\zeta=1/2$ is less than $\zeta_c(0.2)$ and approximately equal to $\zeta_c(0.15)$.)
\label{fig:wake_variance}
}
\end{figure} 

Write $\nu_{L,\zeta}$ for the stationary distribution of the ARW wake chain of $\floor{\zeta L^d}$ particles on the discrete torus $\Z_L^d$. For a stationary configuration $\eta \sim \nu_{L,\zeta}$, we can once again count the number of particles in a box $B = \prod_{i=1}^d [0,\ell_i-1]$ and ask whether the variance is linear or sublinear in the size of $B$.

\begin{question}  \label{q.hyperwake}
Is there $\eps>0$ such that the stationary state $\eta$ of the ARW wake chain satisfies
	\[ \Var_{\eta \sim \nu_{L,\zeta}} (|\eta|_B) = \begin{cases} \Theta(\ell_1\ldots\ell_d), & \zeta<\zeta_c \\
								O((\ell_1\cdots \ell_d)^{1-\eps}), & \zeta \geq \zeta_c ? 
						\end{cases}
	\]
\end{question}

Figure~\ref{fig:wake_variance} tests the case $d=1$ at density $\zeta=1/2$ and two different sleep rates: $\zeta_c(0.2) > 1/2$ and $\zeta_c(0.15) \approx 1/2$.

\section{Site correlations}
\label{s:correlations}

To address Question~\ref{q.measures} (Is $\pi = \mu = \alpha$?) we performed experiments comparing the site correlations in the ARW free chain, ARW wired chain, and ARW point source aggregate. For the free chain on the torus $\Z_L \times \Z_L$ we are able to average with respect to translation and reflection symmetries of the torus to increase the precision of the numerical estimates.  For the wired chain and point source, only reflection symmetries are available, so precision is lower. 

\subsection{Free site correlations} 
For small $x,y$ we computed the empirical correlation coefficient 
	\[ \frac{\EE ( 1_f (0,0)1_f(x,y) )- \zeta^2}{\zeta - \zeta^2} \] 
where $1_f (x,y)$ is the indicator of the event that site $(x,y)$ has a sleeping particle in the stabilization of $\zeta L^2$ particles started at independent uniform random sites on the torus $\Z_{L} \times \Z_{L}$ with side length $L=63$ at sleep rate $\lambda = 2$ and density $\zeta = 0.81 \approx \zeta_c$. We averaged over 80000 independent samples, and over translation and reflection symmetries of the torus. The results are shown in Table~\ref{table.free}.

\subsection{Wired site correlations}
For small $x,y$ we computed the empirical correlation coefficient
	\[ \frac{\EE ( 1_w (0,0) 1_w (x,y)) - \zeta^2}{\zeta - \zeta^2} \]
where $1_w (x,y)$ is the indicator of the event that site $(x,y)$ has a sleeping particle in the stationary state $\Stab(1_V)$ of the ARW wired chain on the square $V = [-31,31]^2 \subset \Z^2$ at sleep rate $\lambda=2$, and $\zeta = 0.81 \approx \zeta_s$.  We averaged over $5 \cdot 10^{7}$ independent samples, and over the $D_8$ symmetry of the square lattice.  The results are shown in Table~\ref{table.wired}.

\subsection{Point source site correlations}

For small $x,y$ we computed the empirical correlation coefficient
	\[ \frac{\EE (1_a (0,0) 1_a (x,y)) - \zeta^2}{\zeta - \zeta^2} \]
where $1_a (x,y)$ is the indicator of the event that site $(x,y)$ has a sleeping particle in the stabilization of $n=3215$ particles started at $(0,0)$, and $\zeta = 0.81 \approx \zeta_a$.  This value of $n$ was chosen to make the total number of particles match the free chain experiment described above. We averaged over $10^5$ independent samples, and over the $D_8$ symmetry of the square lattice. The results are shown in Table~\ref{table.pointsource}.

\begin{table}[ht]
\centering
\begin{tabular}{c|rrrrrr}
$y \backslash x$ & 0 & 1 & 2 & 3 & 4 & 5 \\
\hline
0 & 1.0000 & -0.0238 & -0.0101 & -0.0061 & -0.0045 & -0.0037 \\
1 & & -0.0139 & -0.0082 & -0.0056 & -0.0044 & -0.0037 \\
2 & & & -0.0063 & -0.0048 & -0.0040 & -0.0035 \\
3 & & & & -0.0042 & -0.0037 & -0.0034 \\
4 & & & & & -0.0034 & -0.0032 \\
5 & & & & & & -0.0031 \\
\end{tabular}
\caption{Short-range correlations for the ARW free Markov chain on the discrete torus $\Z_{63} \times \Z_{63}$. \label{table.free}
}
\end{table}

\begin{table}[h]
\centering
\begin{tabular}{c|rrrr}
$y \backslash x$ & 0 & 1 & 2 & 3 \\
\hline
0 & 1.000 & -0.021 & -0.008 & -0.003 \\
1 & & -0.012 & -0.006 & -0.003 \\
2 & & & -0.004 & -0.002 \\
3 & & & & -0.001 \\
\end{tabular}
\caption{Short-range correlations for the ARW wired Markov chain on a square of side length $63$ in $\Z^2$. 
\label{table.wired}
}
\end{table}

\begin{table}[h]
\centering
\begin{tabular}{c|rrrr}
$y \backslash x$ & 0 & 1 & 2 \\
0 & 1.000 &  -0.021 & -0.007 \\
 1 & & -0.010 & -0.006 \\
 2 & & & -0.000
 \end{tabular}
 \caption{Short-range correlations for the ARW point source aggregate of $3215$ particles in $\Z^2$.
 \label{table.pointsource}
 }
 \end{table}

Comparing Tables~\ref{table.free} and~\ref{table.wired}, the spatial decay of correlations is faster in the wired chain than in the free chain. Is this an artifact of the small system size, or are the limiting measures $\mu$ and $\pi$ different?  Comparing Tables~\ref{table.wired} and~\ref{table.pointsource}, the short range correlations are consistent with the measures $\pi$ and $\alpha$ being equal, but the low precision prevents us from conjecturing this confidently.

\subsection{Wake chain site correlations}

\begin{figure}[ht]
\includegraphics[width=.95\textwidth]{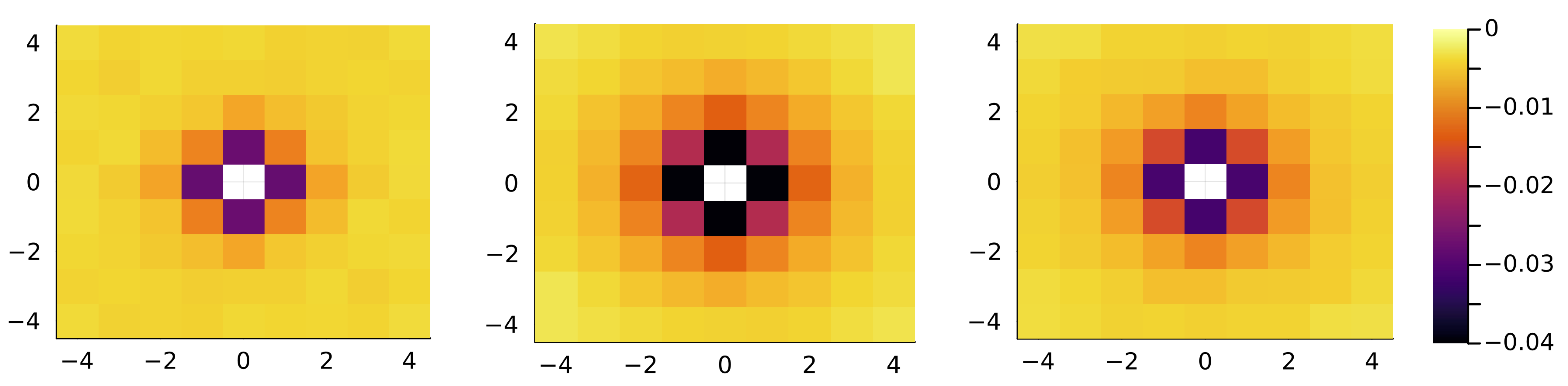}
\begin{tabular}{ccc}
$\qquad \zeta = 0.2 \hspace{1.4cm}$ & $\hspace{1cm}  \zeta = 0.6 \hspace{1.3cm}$ & $\hspace{1cm} \zeta = 0.7 \qquad \qquad $
\end{tabular}
\caption{Short-range correlations of the ARW Wake Markov Chain at three different densities $\zeta = 0.2,0.6,0.7$. Each non-central site is shaded according to the correlation coefficient between the events that a sleeping particle is located at that site and at the central site.
\label{fig.wake2}
}
\end{figure}

The ARW wake chain has a family of stationary distributions, one for each density. We wondered whether the site correlations depend on the density. The answer appears to be yes, as shown in Figure~\ref{fig.wake2}. For three different densities $\zeta \in \{0.2,0.6,0.7\}$ we computed the empirical correlation coefficient
	\[ \frac{\EE ( 1_z (0,0) 1_z (x,y)) - \zeta^2}{\zeta - \zeta^2} \]
where $1_z(x,y)$ is the indicator of the event that site $(x,y)$ has a sleeping particle in the stationary distribution of the ARW wake chain at density $\zeta$ on the torus $\Z_{15} \times \Z_{15}$ with sleep rate $1$. After a burn-in period to reach stationarity, we averaged over $100000$ subsequent time steps of the ARW wake chain and over translation symmetries of the torus. Short-range correlations are negative at all densities, and nothing notable seems to happen at the threshold density $\zeta_c \approx 0.68$.  The strongest nearest-neighbor site correlation occurs at a density somewhat less than $\zeta_c$.

\section{Contrasts with the Abelian sandpile model}
\label{s:contrasts}

We close by comparing Activated Random Walk to the Abelian Sandpile, and in particular highlighting which results, among the ones we believe to hold for ARW, are known to fail for the sandpile model. 

\subsection{Point Source}

Pegden and Smart \cite{PS13} proved existence of a limit shape for the point source Abelian Sandpile in $\Z^d$. The scaling limit of the Abelian sandpile on $\Z^2$ obeys a PDE that is not rotationally symmetric \cite{LPS1, LPS2}, so its limit shape from a point source is unlikely to be a Euclidean disk (although it has not been formally proved not to be a disk!). 
One symptom of the failure of universality in the Abelian Sandpile is the existence of ``dense clumps'' in the point-source sandpile. These are macroscopic regions whose density is higher than the average density of the whole pile. By contrast, we believe Activated Random Walkers are incompressible (Conjecture~\ref{c.compress}).

\subsection{Stationary Ergodic}

The Abelian Sandpile in $\Z^d$ for $d\geq 2$ has an interval of threshold densities: any density between $d$ and $2d-1$ can be threshold, depending on the law of the initial configuration \cite{FMR09,MQ,FLP,FR05}). By contrast, Activated Random Walk at a given sleep rate has a single threshold density (Theorem~\ref{t.RSZ}).

Conjecture~\ref{c.subcritical} fails for the Abelian Sandpile, due to slow mixing: the sandpile stabilization of $\eta_0 + \xi_t$ retains a memory of its initial state $\eta_0$ even as $t \uparrow \zth - \zeta_0$ \cite{threshold}. Terms like ``the self-organized critical state'' result in lot of confusion in the physics literature on the Abelian Sandpile because there are many such states!  One of them, the limit of the uniform recurrent state, is amenable to exact calculations 
\textcolor{blue}{\cite{CS12, KW15, looping, KW16}}, but slow driving from a subcritical state will usually produce a critical state with different properties (e.g.\ different density). 

\subsection{Wired Markov Chain}

Fast mixing of ARW stands in contrast to the slow mixing of the Abelian sandpile, where $t_{mix} = \Theta(L^2 \log L)$ for the wired chain on $(\Z/L\Z)^2$ \cite{HJL} and on $[1,L]^2$ \cite{Hough}.  
This logarithmic factor is responsible for the discrepancy between the stationary density $\zeta_s = 2.125000$ and the threshold density $\zeta_c = 2.125288$ observed in \cite{FLW10a}.

To approach Conjectures \ref{c.stationary}--\ref{c.infinitevolume} and Question \ref{q.measures} it may be useful to find a combinatorial description of the ARW stationary distribution.
An important tool available for the Abelian sandpile, which has no counterpart yet in the ARW setting, is a bijection between recurrent states and spanning trees. This bijection is useful because of the well-developed theory of infinite-volume limits like \eqref{eq:infinitevolume} for trees \cite{BLPS}. The bijection from sandpiles to trees plays a starring role in Athreya and J\'{a}rai's proof that the uniform recurrent sandpile on a finite set $V \subset \Z^d$ has an infinite-volume limit \cite{AJ}, and in J\'{a}rai and Redig's study of infinite-volume sandpile dynamics \cite{JR}. Hutchcroft used the bijection with spanning trees to prove universality results for high-dimensional sandpiles \cite{Hutchcroft}. 

\subsection{Free Markov Chain}

The Hockey Stick Conjecture~\ref{c:hockey} is believed to be false for the Abelian Sandpile due to its slow mixing. There is, however, a weaker conjectured relationship between the sandpile free and wired chains: the threshold time of the free chain coincides with the first time when a macroscopic number of particles exit the wired chain \cite{FLW10b}.

\section*{Acknowledgments}

We thank Ahmed Bou-Rabee, Hannah Cairns, Deepak Dhar, Shirshendu Ganguly, Chris Hoffman, Feng Liang, SS Manna, Pradeep Mohanty, Leonardo Rolla, Vladas Sidoravicius, and Lorenzo Taggi for many inspiring conversations. Thanks to Chris Hoffman for pointing out that Conjecture~\ref{c.stationary} requires a condition on the boundary of $V_n$, and that Conjecture~\ref{c:hockey} requires a condition on the driving.
This project was partly supported by the Funds for joint research Cornell-Sapienza.  LL was partly supported by the NSF grant DMS-1105960 and IAS Von-Neumann Fellowship. We thank Cornell University, Sapienza University, IAS and ICTS-TIFR for their hospitality.

\end{document}